\documentclass[review]{elsarticle}
\usepackage[english]{babel}
\usepackage{amsmath}
\usepackage{amssymb}
\usepackage{fancyhdr} 
\usepackage{bussproofs}
\usepackage{mathrsfs}
\usepackage{leftidx}
\usepackage{amsthm}
\usepackage{amsmath}%
\usepackage{amsfonts}%
\usepackage{amssymb}%
\usepackage{tikz}
\usetikzlibrary{matrix,arrows}
\usetikzlibrary{arrows,%
                petri,%
               topaths}%

\usepackage{enumerate}

\newtheorem{theorem}{Theorem}[section]  
\newtheorem{corollary}[theorem]{Corollary}

\theoremstyle{definition}
\newtheorem{definition}[theorem]{Definition}
\newtheorem{example}{Example}[section]
\newtheorem{proposition}{Proposition}[section]

\newtheorem{property}{Property}[section]

\theoremstyle{definition} 

\newtheorem*{maintheorem*}{Main Theorem}

\begin{document}

\begin{frontmatter}

\title{Fuzzy $\alpha$-cut and related structures}



\author[mymainaddress]{Purbita Jana\corref{mycorrespondingauthor}}
\cortext[mycorrespondingauthor]{Corresponding author}
\ead{purbita\_presi@yahoo.co.in}

\author[mysecondaryaddress]{Mihir .K. Chakraborty}
\ead{mihirc4@gmail.com}
\fntext[myfootnote]{The corresponding author acknowledge Department of Science $\&$ Technology, Government of India for financial support vide reference no. SR/WOS-A/PM-1010/2014 under Women Scientist Scheme to carry out this work.}

\address[mymainaddress]{Department of Pure Mathematics, University of Calcutta}
\address[mysecondaryaddress]{School of Cognitive Science, Jadavpur University}

\begin{abstract}
This paper deals with a new notion called fuzzy $\alpha$-cut and its properties. A notion called localic frame is also introduced. Algebraic structures arising out of the family of fuzzy $\alpha$-cuts have been investigated. It will be seen that this family forms a localic frame. Some significance and usefulness of fuzzy $\alpha$-cuts are discussed.
\end{abstract}


\end{frontmatter}


\section{Introduction}
Though well known now-a-days, we would like to start with a little bit of history. Fuzzy set was first introduced and studied by Lotfi Zadeh \cite{LZ} in 1965, which can be considered in a sense a generalisation of ordinary set. It is well known that in informal set theory a (crisp) set $A$ is considered as a subset of a universal set $U$ and is fully determined by a function from $U$ to $\{0,1\}$ called the characteristic function of $A$ (denoted by $\chi_A$). Whereas a fuzzy set is a function from $U$ to $[0,1]$, and in this case the function is known as membership function. In 1967, J. Goguen \cite{GJ} generalised this notion one step further by considering the function from $U$ to $L$ (a complete lattice) and called it $L$-fuzzy set. Subsequently, there had been many other generalisations of the original proposal of Zadeh \cite{AK1, BW, CE}. In this paper we will consider $L$ as a frame (c.f. Definition \ref{fm}), $1_L$ and $0_L$ being the top and the bottom elements respectively.

 In 1971, Zadeh proposed a representation theorem of  fuzzy sets using the notion of $\alpha$-cuts (c.f. Definition \ref{1'} considering $L$ as [0,1]), known as first decomposition theorem \cite{KY} in the literature. $\alpha$-cuts of a fuzzy set are crisp sets. In this paper we delve into the notion of fuzzy $\alpha$-cut (c.f. Definition \ref{2'}), which was introduced by the present authors in \cite{MP2}. Recently we have noticed that the notion of fuzzy $\alpha$-cut exists in the literature as `level fuzzy sets' introduced in \cite{TR}. In this regard the authors are grateful to the editor of this journal for his valuable advise on an earlier version of this paper.  A fuzzy $\alpha$-cut of a fuzzy set gives a fuzzy subset of the given fuzzy set. We will see that fuzzy $\alpha$-cut of a fuzzy mathematical structure is a fuzzy mathematical substructure.



In this paper we have proved that a family of fuzzy $\alpha$-cuts over a frame  forms a localic frame. Consequently they generate a model of so called fuzzy geometric logic with graded consequence \cite{MP3}. The notion of graded frame and fuzzy geometric logic with graded consequence was introduced in \cite{MP3}. Such an algebraic structure and a logic were proposed to serve the purpose of giving an answer to the question-``From which logic fuzzy topology can be generated?" This question came up parallel to a similar idea provided in Vickers's book \cite{SV} ``Topology via logic", viz. from which logic classical topology can be generated?. It is to be noted that as an answer, fuzzy geometric logic was invented and as a further generalisation fuzzy geometric logic with graded consequence was introduced. Notion of graded frame came into the picture as Lindenbaum type algebra of fuzzy geometric logic with graded consequence \cite{MP3}. The notion of localic frame (c.f. Definition \ref{gfl}) is introduced here which is a further generalisation of graded frame by taking a frame-valued binary relation instead of $[0,1]$ -valued binary relation as $[0,1]$ is a particular frame. 

As for usefulness and significance of the notion of fuzzy $\alpha$-cuts, we claim that they give natural substructures of fuzzy topological spaces and fuzzy algebraic structures. Classical fuzzy topological spaces are defined by taking a crisp set and fuzzy open sets. There are two major streams of research, one following Chang's definition \cite{CL} and the other following Lowen's definition \cite{RL}. In both cases, subspaces are defined on crisp subsets of the base set. On the other hand in case of fuzzy algebraic structures (e.g. Rosenfeld \cite{AR}), one starts with a classical algebraic structure and defines fuzzy substructures. In neither construction topological or algebraic, the starting base set is taken to be fuzzy. While in the topological case, the base set as well as all the subspaces are to be taken crisp, in the algebraic case though the sub algebras are fuzzy, one has to begin with a classical crisp algebraic structure.

In \cite{MA} and \cite{MB} there had been proposals to develop both the kinds of structures on fuzzy sets. Besides, in \cite{MB} the algebraic compositions also have been fuzzy right from the start. Fuzzy $\alpha$-cuts being fuzzy subsets of a fuzzy set, using the above proposals it would be possible to define topological and algebraic substructures on them. These will be quite natural substructures. We shall present these constructions in section 4.

Secondly, fuzzy $\alpha$-cuts will provide fuzzy sets as lower and upper approximations in the probabilistic rough set framework \cite{YY}. We expects these kinds of approximation will be useful in the domain of application of rough set theory \cite{PW}. 

This paper is organised as follows. Section 2 emphasises upon the properties of fuzzy $\alpha$-cut. In section 3, various algebraic structures and a notion of G$\ddot{o}$del-like arrow along with its properties are discussed. Algebraic structures formed by the family of fuzzy $\alpha$-cuts are also studied in this section. The significance of fuzzy $\alpha$-cuts is provided in section 4. Section 5 presents some concluding remarks.

We give below some preliminary definitions that would be required in the sequel.
\begin{definition}[Frame]\label{fm}
A \textbf{frame} \index{frame} is a complete lattice such that, $$x\wedge\bigvee Y=\bigvee\{x\wedge y :y\in Y\}.$$
i.e., the binary meet distributes over arbitrary join.
\end{definition}
\begin{definition}[$\alpha$-cut of a fuzzy set]\label{1'}
Let $(X,\tilde{A})$ be an $L$-fuzzy set, where $X$ is the base set and $L$ is a frame. Then for $\alpha \in L$, the \textbf{$\alpha$-cut} of 
$(X,\tilde{A})$ is the ordinary set $\{ x\in X\mid \tilde{A}(x)\geq \alpha \}$ and is denoted by $\alpha_{\tilde{A}}$.
\end{definition} 

\begin{definition}[G$\ddot{o}$del arrow]\cite{KY}
G$\ddot{o}$del arrow is defined as follows:
\begin{center}
$a\rightarrow b$ $=\begin{cases}
        1 & \emph{if $a\leq b$} \\
        b & \emph{if $a>b$}.
    \end{cases}$
    \end{center}
for all $a,b\in[0,1]$.
\end{definition}
G$\ddot{o}$del arrow can be generalised to the following.
\begin{definition}[G$\ddot{o}$del-like arrow]\label{gga}
Let $L$ be any frame. Then the G$\ddot{o}$del-like arrow is defined as follows:\\
\begin{center}
$a\rightarrow b$ $=\begin{cases}
        1_L & \emph{if $a\leq b$} \\
        b & \emph{otherwise}.
    \end{cases}$
\end{center}
for all $a,b\in L$.
\end{definition}
There is another kind of implication in $L$ called residuated implication defined as below.
\begin{definition}[Residuated implication]\label{gha}
Let $L$ be any any frame. Then the residuated implication is defined by $a\rightarrow b=sup\{c\in L\mid c\wedge a\leq b\}$ for all $a,b\in L$. 
\end{definition}
The relationship between these two types of implications is discussed in subsection \ref{pga}.
\section{Fuzzy $\alpha$-cut and its properties}
In this section we will define the notion of fuzzy $\alpha$-cut and provide some of the algebraic properties of fuzzy $\alpha$-cut. For the corresponding classical notion we refer to \cite{KY}. 

\begin{definition}[Fuzzy $\alpha$-cut of a fuzzy set]\label{2'}\cite{MP2}
Let $(X,\tilde{A})$ be an $L$-fuzzy set. Then for $\alpha \in L$, the \textbf{fuzzy $\alpha$-cut} of 
$(X,\tilde{A})$ is the fuzzy subset $(X,\leftidx{^{\tilde{A}}}{\alpha}{})$ where $\leftidx{^{\tilde{A}}}{\alpha}{}$ is defined as follows:
\begin{center}
$\leftidx{^{\tilde{A}}}{\alpha}{}(x)$ $=\begin{cases}
        \tilde{A}(x) & \emph{if $\tilde{A}(x)\geq \alpha $} \\
        0_L & \emph{otherwise}.\end{cases}$
\end{center}  
We will denote fuzzy $\alpha$-cut of an $L$-fuzzy set $(X,\tilde{A})$ simply by $\leftidx{^{\tilde{A}}}{\alpha}{}$ if the base set $X$ is understood.   
\end{definition} 
It is to be noted that fuzzy $\alpha$-cut of a fuzzy set is also known as level fuzzy set \cite{TR}. As mentioned in the introduction, present authors were not aware of this paper published long back in 1977 and not used frequently in subsequent literature. In the paper \cite{TR} the author defined the algebraic operations viz. intersection, union, complementation of level sets as is done in fuzzy set theory by min, max and $1-(\cdot)$, in the value set $[0,1]$ and established certain elementary properties. In our paper, however, these operations are presumed since these are none else than the corresponding operations of fuzzy subsets. We have rather proved some non-trivial results in this section where the value set $L$ is taken to be a frame. 
\begin{example}
Consider the fuzzy set $\tilde {A}$ defined on the interval $X=[0,10]$ of real numbers by the membership function $\tilde{A}(x)=\frac{x}{x+2}$. Then 
\begin{center}
$\leftidx{^{\tilde{A}}}{0.2}{}(x)$ $=\begin{cases}
        \frac{x}{x+2} & \emph{for $x\in [0.5,10] $} \\
        0 & \emph{otherwise}.\end{cases}$
\end{center} 
\end{example}
Let $(X,\tilde{A})$, $(X,\tilde{B})$ be two $L$-fuzzy sets. Then $\tilde{A}\subseteq \tilde{B}$ if and only if $\tilde{A}(x)\leq \tilde{B}(x)$, for any $x\in X$.

 As we are dealing with $L$-fuzzy sets, we are considering the generalised version of the definition of $\alpha$-cut and fuzzy $\alpha$-cut by generalising the value set $[0,1]$ to a frame $L$. 

\begin{theorem} Let $(X,\tilde{A})$, $(X,\tilde{B})$ be two fuzzy sets. Then for any $\alpha$, $\alpha_1$, $\alpha_2\in L$ the following properties hold:
1. $\leftidx{^{\tilde{A}}}{\alpha}{}\subseteq \chi_{\alpha_{\tilde{A}}}$; 2. $\alpha_1\leq \alpha_2$ implies $\leftidx{^{\tilde{A}}}{\alpha_1}{}\supseteq \leftidx{^{\tilde{A}}}{\alpha_2}{}$; \ \ 3.  $\leftidx{^{(\tilde{A}\cap\tilde{B})}}{\alpha}{}=\leftidx{^{\tilde{A}}}{\alpha}{}\cap \leftidx{^{\tilde{B}}}{\alpha}{}$ and $\leftidx{^{(\tilde{A}\cup\tilde{B})}}{\alpha}{}=\leftidx{^{\tilde{A}}}{\alpha}{}\cup \leftidx{^{\tilde{B}}}{\alpha}{}$.
\end{theorem}
\begin{proof} We demonstrate the proof of 3 only. For the first part of 3 we proceed as follows:

\begin{align*}
\leftidx{^{(\tilde{A}\cap\tilde{B})}}{\alpha}{}(x) & =\begin{cases}
        (\tilde{A}\cap\tilde{B})(x) & \emph{if $(\tilde{A}\cap\tilde{B})(x)\geq \alpha $} \\
        0_L & \emph{otherwise}.\end{cases}\\ 
        & =\begin{cases}
        \tilde{A}(x)\wedge\tilde{B}(x) & \emph{if $\tilde{A}(x)\wedge\tilde{B}(x)\geq \alpha $} \\
        0_L & \emph{otherwise}.\end{cases}\\
        & =\begin{cases}
        \tilde{A}(x)\wedge\tilde{B}(x) & \emph{if $\tilde{A}(x)\geq \alpha$ and $\tilde{B}(x)\geq \alpha $} \\
        0_L & \emph{otherwise}.\end{cases}\\
        & =\begin{cases}
        \tilde{A}(x)\wedge\tilde{B}(x) & \emph{if $\tilde{A}(x)\geq \alpha$ and $\tilde{B}(x)\geq \alpha $} \\
        0_L & \emph{if $\tilde{A}(x)\geq \alpha$ and $\tilde{B}(x)< \alpha $} \\
        0_L & \emph{if $\tilde{A}(x)< \alpha$ and $\tilde{B}(x)\geq \alpha $} \\
        0_L & \emph{if $\tilde{A}(x)< \alpha$ and $\tilde{B}(x)< \alpha $}.\end{cases}\\
        & =\begin{cases}
        \tilde{A}(x)\wedge\tilde{B}(x) & \emph{if $\tilde{A}(x)\geq \alpha$ and $\tilde{B}(x)\geq \alpha $} \\
        \tilde{A}(x)\wedge 0_L & \emph{if $\tilde{A}(x)\geq \alpha$ and $\tilde{B}(x)< \alpha $} \\
        0_L\wedge\tilde{B}(x) & \emph{if $\tilde{A}(x)< \alpha$ and $\tilde{B}(x)\geq \alpha $} \\
        0_L & \emph{if $\tilde{A}(x)< \alpha$ and $\tilde{B}(x)< \alpha $}.\end{cases}\\
 & = \leftidx{^{\tilde{A}}}{\alpha}{}(x)\wedge\leftidx{^{\tilde{B}}}{\alpha}{}(x)
       =(\leftidx{^{\tilde{A}}}{\alpha}{}\cap\leftidx{^{\tilde{B}}}{\alpha}{})(x).       
\end{align*}  
Similarly the second equality holds.
\end{proof}
For any mapping $f:X\longrightarrow Y$, the image of the fuzzy subset $(X,\tilde{A})$ of $X$ is the fuzzy subset $(Y,f(\tilde{A}))$ and defined by \cite{KY}
$$f(\tilde{A})(y)=\bigvee_{x\in X}\{\tilde{A}(x)\mid y=f(x)\}.$$
Thus $f(\leftidx{^{\tilde{A}}}{\alpha}{})$ gives the fuzzy subset $(Y,\leftidx{^{\tilde{A}}}{\alpha}{})$ of $Y$. We now have the following theorem.
 \begin{theorem}
 Let $f:X\longrightarrow Y$ be a mapping. Then for any $L$-fuzzy set $(X,\tilde{A})$ and $\alpha\in L$, $f(\leftidx{^{\tilde{A}}}{\alpha}{})= \leftidx{^{(f(\tilde{A}))}}{\alpha}{}$.
 \end{theorem}
 \begin{proof}
 For any $y\in Y$, we have the following:
 \begin{align*}
 f(\leftidx{^{\tilde{A}}}{\alpha}{})(y)  & =\bigvee_x\{\leftidx{^{\tilde{A}}}{\alpha}{}(x)\mid y=f(x)\} \\
& =\begin{cases}
        0_L & \emph{if $\tilde{A}(x)< \alpha $ for all $x\in X$} \\
        \bigvee_x\{\tilde{A}(x)\mid y=f(x)\} & \emph{otherwise}.\end{cases}\\
        & =\begin{cases}
        \bigvee_x\{\tilde{A}(x)\mid y=f(x)\} & \emph{if $\bigvee_x\{\tilde{A}(x)\mid y=f(x)\}\geq \alpha $} \\
        0_L & \emph{otherwise}.\end{cases}\\
         & =\begin{cases}
       (f(\tilde{A}))(y) & \emph{if $(f(\tilde{A}))(y)\geq \alpha $} \\
        0_L & \emph{otherwise}.\end{cases}\\
        & =\leftidx{^{(f(\tilde{A}))}}{\alpha}{}(y).
 \end{align*}
 Hence $f(\leftidx{^{\tilde{A}}}{\alpha}{})= \leftidx{^{(f(\tilde{A}))}}{\alpha}{}$.
 \end{proof}
 \begin{proposition}
 Let $f:X\longrightarrow Y$ be a mapping. Then for any $L$-fuzzy set $(X,\tilde{A})$ and $\alpha, \beta\in L$, $\alpha\leq \beta\Rightarrow f(\leftidx{^{\tilde{A}}}{\alpha}{})\subseteq f(\leftidx{^{\tilde{A}}}{\beta}{})$.
 \end{proposition}
 \begin{proposition}
 Let $f:X\longrightarrow Y$ and $g:Y\longrightarrow Z$. Then for any $L$-fuzzy set $(X,\tilde{A})$ and $\alpha\in L$, $(g\circ f)(\leftidx{^{\tilde{A}}}{\alpha}{})=\leftidx{^{(g\circ f(\tilde{A}))}}{\alpha}{}$.
 \end{proposition}
 \begin{proof}
 $(g\circ f)(\leftidx{^{\tilde{A}}}{\alpha}{})=g(f(\leftidx{^{\tilde{A}}}{\alpha}{}))=g(\leftidx{^{(f(\tilde{A}))}}{\alpha}{})=\leftidx{^{(g(f(\tilde{A})))}}{\alpha}{}=\leftidx{^{(g\circ f(\tilde{A}))}}{\alpha}{}$.
\end{proof}
$(Z,(g\circ f)\leftidx{^{\tilde{A}}}{\alpha}{})$ is a fuzzy subset of $Z$.
\begin{proposition}
 Let $f:X\longrightarrow Y$, $g:Y\longrightarrow Z$ and $h:Z\longrightarrow W$. Then for any $L$-fuzzy set $(X,\tilde{A})$ and $\alpha\in L$, $(h\circ(g\circ f))(\leftidx{^{\tilde{A}}}{\alpha}{})=((h\circ g)\circ f)(\leftidx{^{\tilde{A}}}{\alpha}{})$.
 \end{proposition}
 \begin{proposition}
 Let $f:X\longrightarrow Y$ be a mapping and $id_X:X\longrightarrow X$ be the identity mapping. Then for any $L$-fuzzy set $(X,\tilde{A})$ and $\alpha\in L$, $(f\circ id_X)(\leftidx{^{\tilde{A}}}{\alpha}{})=f(\leftidx{^{\tilde{A}}}{\alpha}{})$. 
 \end{proposition}
 \begin{proposition}
 Let $f:X\longrightarrow Y$ be a mapping and $id_Y:Y\longrightarrow Y$ be the identity mapping. Then for any $L$-fuzzy set $(X,\tilde{A})$ and $\alpha\in L$, $(id_Y\circ f)(\leftidx{^{\tilde{A}}}{\alpha}{})=f(\leftidx{^{\tilde{A}}}{\alpha}{})$. 
 \end{proposition}
 \begin{proof}
 $\alpha\leq\beta\Rightarrow \leftidx{^{\tilde{A}}}{\alpha}{}\subseteq \leftidx{^{\tilde{A}}}{\beta}{}\Rightarrow f(\leftidx{^{\tilde{A}}}{\alpha}{})\subseteq f(\leftidx{^{\tilde{A}}}{\beta}{})$.
 \end{proof}
\section{Algebraic structure of the family of fuzzy $\alpha$-cuts}
We shall establish that the family of fuzzy $\alpha$-cuts forms a localic frame [c.f. Definition \ref{gfl}] with respect to an $L$-fuzzy relation $R$ defined in terms of the G$\ddot{o}$del-like arrow [c.f. Definition \ref{gga}] in $L$. 
\begin{definition}[Localic Frame]\label{gfl}\index{graded frame}
A \textbf{localic frame} is a 5-tuple $(A,\top,\wedge,\bigvee,R_L)$, where $A$ is a non-empty set, $\top\in A$, $\wedge$ is a binary operation, $\bigvee$ is an operation on arbitrary subset of $A$, $R_L$ is an $L$-valued fuzzy binary relation on $A$ satisfying the following conditions:
\begin{enumerate}
\item $R_L(a,a)=1_L$ (fuzzy reflexivity);
\item $R_L(a,b)=1_L=R_L(b,a)\Rightarrow a=b$ (fuzzy antisymmetry);
\item $R_L(a,b)\wedge R_L(b,c)\leq R_L(a,c)$ (fuzzy transitivity);
\item $R_L(a\wedge b,a)=1_L=R_L(a\wedge b,b)$;
\item $R_L(a,\top)=1_L$;
\item $R_L(a,b)\wedge R_L(a,c)=R_L(a,b\wedge c)$;
\item $R_L(a,\bigvee S)=1_L$ if $a\in S$;
\item $inf\{R_L(a,b)\mid a\in S\}=R_L(\bigvee S,b)$;
\item $R_L(a\wedge\bigvee S,\bigvee\{a\wedge b\mid b\in S\})=1_L$;
\end{enumerate}
for any $a,b,c\in A$ and $S\subseteq A$. We will denote a localic frame by $(A,R_L)$. 
\end{definition}
In particular, $(A,\top,\wedge,\bigvee,R_{[0,1]})$ is a graded frame \cite{MP3}. Thus localic frame is a generalisation of graded frame. It is to be noted that an algebraic structure satisfying the first three conditions of Definition \ref{gfl} is known as localic poset \cite{JA}. That is, a localic poset is a set endowed with fuzzy partial order relation. However, this is not the only definition of fuzzy partial order. For more general definitions see \cite{MS}.
\begin{theorem}\label{theorem4}
$(\{\leftidx{^{\tilde{A}}}{\alpha}{}\mid \alpha \in L\},\subseteq, \cap, \bigcup)$ is a frame.
\end{theorem}
\begin{proof}
Here we only show the distributive property i.e., $$\leftidx{^{\tilde{A}}}{\alpha}{}\cap\bigcup_i\leftidx{^{\tilde{A}}}{\alpha_i}{}=\bigcup_i(\leftidx{^{\tilde{A}}}{\alpha}{}\cap\leftidx{^{\tilde{A}}}{\alpha_i}{}).$$
\begin{align*}
(\leftidx{^{\tilde{A}}}{\alpha}{}\cap\bigcup_i\leftidx{^{\tilde{A}}}{\alpha_i}{})(x) & = \leftidx{^{\tilde{A}}}{\alpha}{}(x)\wedge((\bigcup_i\leftidx{^{\tilde{A}}}{\alpha_i}{})(x))\\
& = \leftidx{^{\tilde{A}}}{\alpha}{}(x)\wedge(\bigvee_i\leftidx{^{\tilde{A}}}{\alpha_i}{}(x))\\
& = \bigvee_i(\leftidx{^{\tilde{A}}}{\alpha}{}(x)\wedge\leftidx{^{\tilde{A}}}{\alpha_i}{}(x))\ \text{[as $L$ is a frame]} \\
& = \bigvee_i(\leftidx{^{\tilde{A}}}{\alpha}{}\cap\leftidx{^{\tilde{A}}}{\alpha_i}{})(x))\\
& = (\bigcup_i(\leftidx{^{\tilde{A}}}{\alpha}{}\cap\leftidx{^{\tilde{A}}}{\alpha_i}{}))(x).
\end{align*}
This completes the proof.
\end{proof}

\subsection{Prelinear and Semilinear Frame}\label{PSF}
In this subsection we will consider prelinear frame and semilinear frame. This subsection includes detailed study of the above mentioned notions with examples.
\begin{definition}[Prelinear Frame]\label{pfm}\cite{RB}
A frame $L$ together with a binary operation $\rightarrow$ is said to be a prelinear if for each $l_1,l_2\in L$, $(l_1\rightarrow l_2)\vee (l_2\rightarrow l_1)=\top$, where $\top$ is the top element of $L$.
\end{definition}
For our purpose we will take $\rightarrow$ as G$\ddot{o}$del-like arrow. In this section henceforth all the arrows are G$\ddot{o}$del-like arrow.

\underline{\textbf{Note:}} For the notion of prelinearity in more general set up we refer
to \cite{PH}. Whenever there is an $\rightarrow$ satisfying the property $l_1\leq l_2$ 
implies $l_1\rightarrow l_2=\top$, linearity of the order implies prelinearity. Our purpose here will be served under the assumption of a notion more general than prelinearity viz. semilinearity (c.f. Definition \ref{sfm}). The purpose is to generalize the notion of graded frame \cite{MP3}  where the value set is taken as $[0,1]$. 
\begin{definition}[Semilinear Frame]\label{sfm}
A semilinear frame $L=(L,\wedge,\bigvee,\rightarrow)$ is a frame $(L,\wedge,\bigvee)$ together with a binary operation $\rightarrow$ such that for all $l_1,l_2,l_3\in L$, $(l_1\rightarrow l_2)\wedge (l_1\rightarrow l_3)= (l_1\rightarrow l_2\wedge l_3)$.
\end{definition}
It can be verified by considering all possible cases that any frame with up to 4-elements is always preilinear.
We give below an example of a 5-element lattice which is the smallest semilinear but not prelinear frame.
\begin{example}\label{snp}
The following frame is not prelinear but semilinear.
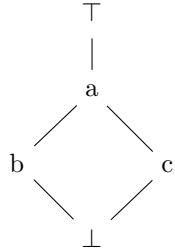
\begin{center}
\begin{tikzpicture}
  \matrix (galois)
     [matrix of nodes,%
      nodes in empty cells,
      nodes={outer sep=0pt,circle,minimum size=4pt},
      column sep={1cm,between origins},
      row sep={1cm,between origins}]
   {
    && $\top$ &&\\
    &  & a & \\
    & b &  & c\\
    && $\bot$ &&\\
   };
      \draw (galois-1-3) -- (galois-2-3);
      \draw (galois-2-3) -- (galois-3-2);
      \draw (galois-2-3) -- (galois-3-4);
       \draw (galois-3-2) -- (galois-4-3);
       \draw (galois-3-4) -- (galois-4-3);

\end{tikzpicture} 
\end{center}
For this frame $(b\rightarrow c)\vee (c\rightarrow b)=c\vee b=a\neq\top$. Hence it is not prelinear.  
\end{example}

The following is an example of a frame with six elements which is not semilinear.
It is to be noted that the following frame is the smallest non semilinear frame. In other words a non-semilinear distributive lattice contains at least six elements.

\begin{example}
The following frame is not semilinear.
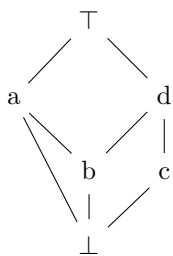
\begin{center}
\begin{tikzpicture}
  \matrix (galois)
     [matrix of nodes,%
      nodes in empty cells,
      nodes={outer sep=0pt,circle,minimum size=4pt},
      column sep={1cm,between origins},
      row sep={1cm,between origins}]
   {
    && $\top$ &&\\
    & a &  & d\\
     &  & b & c\\
    && $\bot$ &&\\
   };
     \draw (galois-1-3) -- (galois-2-2);
     \draw (galois-1-3) -- (galois-2-4);
      \draw (galois-2-2) -- (galois-3-3);
       \draw (galois-2-4) -- (galois-3-3);
      \draw (galois-2-4) -- (galois-3-4);
       \draw (galois-2-2) -- (galois-4-3);
      \draw (galois-3-3) -- (galois-4-3);
       \draw (galois-3-4) -- (galois-4-3);

\end{tikzpicture} 
\end{center}  
For this frame $(b\rightarrow a)\wedge (b\rightarrow c)=\top\wedge c=c$, whereas $b\rightarrow (a\wedge c)=b\rightarrow \bot=\bot$. So, Property \ref{prop5} fails.
\end{example} 
\begin{example}
The following frame is the smallest Boolean algebra which is not semilinear.
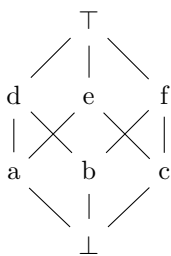
\begin{center}
\begin{tikzpicture}
  \matrix (galois)
     [matrix of nodes,%
      nodes in empty cells,
      nodes={outer sep=0pt,circle,minimum size=4pt},
      column sep={1cm,between origins},
      row sep={1cm,between origins}]
   {
    && $\top$ &&\\
    & d & e & f\\
     & a & b & c\\
    && $\bot$ &&\\
   };
     \draw (galois-1-3) -- (galois-2-2);
      \draw (galois-1-3) -- (galois-2-3);
     \draw (galois-1-3) -- (galois-2-4);
     \draw (galois-2-2) -- (galois-3-2);
      \draw (galois-2-2) -- (galois-3-3);
      \draw (galois-2-3) -- (galois-3-2);
      \draw (galois-2-3) -- (galois-3-4);
       \draw (galois-2-4) -- (galois-3-3);
      \draw (galois-2-4) -- (galois-3-4);
       \draw (galois-3-2) -- (galois-4-3);
      \draw (galois-3-3) -- (galois-4-3);
       \draw (galois-3-4) -- (galois-4-3);

\end{tikzpicture} 
\end{center}  
For this frame $(a\rightarrow c)\wedge (a\rightarrow d)=c\wedge \top=c$, whereas $a\rightarrow (c\wedge d)=a\rightarrow \bot=\bot$. So, Property \ref{prop5} fails.
\end{example}

One can see that the concepts of prelinearity and semilinearity are based on the underlying lattice of the frame which is distributive. While prelinearity is an well known concept, semilinearity is not so and which is a more general concept [c.f. Property \ref{prop5'}]

Before proceeding to the next theorem let us enlist below some properties of G$\ddot{o}$del-like arrow \cite{KY} that would be used in the sequel. 
\subsection{Properties of G$\ddot{o}$del-like arrow}\label{pga}
In this subsection some required properties of G$\ddot{o}$del-like arrow are listed along with verification of some of them.
\begin{property}\label{prop1}
$a\rightarrow a=1_L$, for any $a\in L$.
\end{property}
\begin{property}\label{prop2}
$(a\rightarrow b)\wedge (b\rightarrow c)\leq (a\rightarrow c)$, for any $a,b,c\in L$.
\end{property}
\begin{proof}
It may be observed that the values of $a\rightarrow b$ is either $1_L$ or $b$. Similarly for $b\rightarrow c$ the values are either $1_L$ or $c$ and for $a\rightarrow c$ values are either $1_L$ or $c$. Now the following cases may arise:\\
Case 1: $a\rightarrow b=1_L$ and $b\rightarrow c=1_L$.\\
Here $a\leq b$ and $b\leq c$ and hence as $L$ is transitive, $a\leq c$. Consequently $a\rightarrow c=1_L$. Therefore $(a\rightarrow b)\wedge (b\rightarrow c)=1_L\wedge 1_L=1_L=(a\rightarrow c)$.\\
Case 2: $a\rightarrow b=1_L$ and $b\rightarrow c=c$.\\
We have $(a\rightarrow b)\wedge (b\rightarrow c)=1_L\wedge c=c\leq c\ (or\ 1_L)=(a\rightarrow c)$.\\
Case 3: $a\rightarrow b=b$ and $b\rightarrow c=1_L$.\\
As $b\leq c$,  $(a\rightarrow b)\wedge (b\rightarrow c)=b\wedge 1_L=b\leq c\ (or\ 1_L)=(a\rightarrow c)$.\\
Case 4: $a\rightarrow b=b$ and $b\rightarrow c=c$.\\
In this case $(a\rightarrow b)\wedge (b\rightarrow c)=b\wedge c=c\leq c\ (or\ 1_L)=(a\rightarrow c)$.
\end{proof}
\begin{property}\label{prop3}
 $a\leq b$ implies $(a\rightarrow x)\geq (b\rightarrow x)$, for any $a,b,x\in L$.
 \end{property}
\begin{property}\label{prop4}
$a\leq b$ implies $(x\rightarrow a)\leq (x\rightarrow b)$, for any $a,b,x\in L$.
\end{property}
\begin{property}\label{prop5'}
If $L$ is prelinear then it is semilinear.
\end{property}
\begin{proof}
If possible let  $L$ is prelinear i.e., $(a\rightarrow b)\vee (b\rightarrow a)=\top$, for any $a,b\in L$ and for some $a,b,c\in L$, $(a\rightarrow b)\wedge (a\rightarrow c)\neq a\rightarrow (b\wedge c)$. Then two cases may arise. Case 1: $a< b$, $(a,c)$ and $(b,c)$ are incomparable [where $(a,b)$ represents the pair of points from $L$]. Case 2: $a< c$, $(a,b)$ and $(b,c)$ are incomparable. 

Case 1: In this case notice that $a\vee c=\top$ as $(c\rightarrow a)\vee (a\rightarrow c)=\top$ and $a,c$ are incomparable. Hence $b\wedge(c\vee a)=b\wedge \top=b$.
Now $(b\wedge c)\vee (b\wedge a)=(b\wedge c)\vee a$. The following three cases may arise under this situation. Either $b\wedge c\leq a$ or $b\wedge c>a$ or the pair $(b\wedge c,a)$ is incomparable. If $b\wedge c\leq a$, then $(b\wedge c)\vee a=a\neq b$. When $b\wedge c>a$ then $(b\wedge c)\vee a=b\wedge c\neq b$ as if $b\wedge c=b$ then $b\leq c$, a contradiction. As $L$ is prelinear $(a\rightarrow(b\wedge c))\vee ((b\wedge c)\rightarrow a)=\top$. When the pair $(b\wedge c,a)$ is incomparable then $(a\rightarrow(b\wedge c))\vee ((b\wedge c)\rightarrow a)=(b\wedge c)\vee a=\top\neq b$ (as $b$ and $c$ are incomparable).

Hence for either cases $b\wedge(c\vee a)=b\neq (b\wedge c)\vee (b\wedge a)$, but $L$ is distributive.

Similarly for Case 2 also we get a contradiction.
\end{proof}
\begin{corollary}\label{prop5}
 If $L$ is totally ordered frame then $(a\rightarrow b)\wedge (a\rightarrow c)=a\rightarrow (b\wedge c)$, for any $a,b,c\in L$.
 \end{corollary}
   


\begin{property}\label{prop6}
$inf_i\{(a_i\rightarrow b)\}=sup_i\{a_{i}\}\rightarrow b$, for any $a_i,b\in L$.
\end{property}
\begin{proof}
$sup_i\{a_i\}\rightarrow b=\begin{cases}
        1 & \text{if}\ \ sup_i\{a_i\}\leq b \\
        b & \text{otherwise}.
    \end{cases}$\\
Now for $sup_i\{a_i\}\leq b$ we have $a_i\leq sup_i\{a_i\}\leq b$.\\
Hence for this case $(a_i\rightarrow b)=1$, for each $i$ and consequently $inf_i\{a_i\rightarrow b\}=1$.\\
If $sup_i\{a_i\}>b$ then there exist atleast one $a_i$ such that $a_i>b$ and rest will be either bellow $b$ or equal to $b$. Now for the case $a_i>b$, $a_i\rightarrow b=b$ and for all other cases $a_i\rightarrow b=1$. As $b\leq 1$, $inf_i\{a_i\rightarrow b\}=b$. If $sup\{a_i\}$ and $b$ are incomparable then atleast one of the $a_i$'s, say $a_j$ is incomparable to $b$ and consequently $a_j\rightarrow b=b$. Hence $inf_i\{a_i\rightarrow b\}=b$.
\end{proof}
\begin{property}\label{prop7}
$a\leq b$ iff $a\rightarrow b=1_L$.
\end{property}
\begin{property}\label{prop8}
 $a\wedge (a\rightarrow b)\leq b$.
 \end{property}
 It is to be noted that all the above properties are true for generalised G$\ddot{o}$del arrow as these properties are satisfied by any  arrow with residuation property. That means residuated arrow satisfies semilinear property but there are semilinear arrows which are not residuated arrows.
We now proceed to the main theorem which constitutes subsection \ref{TH}
\subsection{Main Theorem}\label{TH}
\begin{theorem}
Let $L$ be semilinear frame. Then
$(\{\leftidx{^{\tilde{A}}}{\alpha}{}\mid \alpha \in L\},\leftidx{^{\tilde{A}}}{0_L}{}, \cap, \bigcup,R_L)$ is a localic frame, where $R_L(\leftidx{^{\tilde{A}}}{\alpha_1}{}, \leftidx{^{\tilde{A}}}{\alpha_2}{})=inf_x\{\leftidx{^{\tilde{A}}}{\alpha_1}{}(x)\rightarrow \leftidx{^{\tilde{A}}}{\alpha_2}{}(x)\}$ for $\alpha_1$, $\alpha_2\in L$ and `$\rightarrow$' is the G$\ddot{o}$del-like arrow.
\end{theorem}
\begin{proof}
Let us check the properties for $(\{\leftidx{^{\tilde{A}}}{\alpha}{}\mid \alpha \in L\},\leftidx{^{\tilde{A}}}{0_L}{}, \cap, \bigcup,R_L)$ to be a localic frame.
\begin{enumerate}
\item $R_L(\leftidx{^{\tilde{A}}}{\alpha}{},\leftidx{^{\tilde{A}}}{\alpha}{})=inf_x\{\leftidx{^{\tilde{A}}}{\alpha}{}(x)\rightarrow \leftidx{^{\tilde{A}}}{\alpha}{}(x)\}=1_L$ [from Property \ref{prop1} $\leftidx{^{\tilde{A}}}{\alpha}{}(x)\rightarrow\leftidx{^{\tilde{A}}}{\alpha}{}(x)=1_L$, for all $x$].

\item Let $R_L(\leftidx{^{\tilde{A}}}{\alpha_1}{},\leftidx{^{\tilde{A}}}{\alpha_2}{})=1_L=R_L(\leftidx{^{\tilde{A}}}{\alpha_2}{},\leftidx{^{\tilde{A}}}{\alpha_1}{})$. So, $inf_x\{\leftidx{^{\tilde{A}}}{\alpha_1}{}(x)\rightarrow \leftidx{^{\tilde{A}}}{\alpha_2}{}(x)\}=1_L=inf_x\{\leftidx{^{\tilde{A}}}{\alpha_2}{}(x)\rightarrow \leftidx{^{\tilde{A}}}{\alpha_1}{}(x)\}$. Therefore $\leftidx{^{\tilde{A}}}{\alpha_1}{}(x)\leq\leftidx{^{\tilde{A}}}{\alpha_2}{}(x)$ and $\leftidx{^{\tilde{A}}}{\alpha_2}{}(x)\leq\leftidx{^{\tilde{A}}}{\alpha_1}{}(x)$, for all $x$. So, $\leftidx{^{\tilde{A}}}{\alpha_1}{}(x)=\leftidx{^{\tilde{A}}}{\alpha_2}{}(x)$, for any $x$. Hence $\leftidx{^{\tilde{A}}}{\alpha_1}{}=\leftidx{^{\tilde{A}}}{\alpha_2}{}$.

\item From Property \ref{prop2}, we have $(\leftidx{^{\tilde{A}}}{\alpha_1}{}(x)\rightarrow \leftidx{^{\tilde{A}}}{\alpha_2}{}(x))\wedge (\leftidx{^{\tilde{A}}}{\alpha_2}{}(x)\rightarrow \leftidx{^{\tilde{A}}}{\alpha_3}{}(x))\leq (\leftidx{^{\tilde{A}}}{\alpha_1}{}(x)\rightarrow\leftidx{^{\tilde{A}}}{\alpha_3}{}(x))$, for all $x$. Hence $inf_x\{(\leftidx{^{\tilde{A}}}{\alpha_1}{}(x)\rightarrow\leftidx{^{\tilde{A}}}{\alpha_2}{}(x))\wedge(\leftidx{^{\tilde{A}}}{\alpha_2}{}(x)\rightarrow\leftidx{^{\tilde{A}}}{\alpha_3}{}(x))\}\leq (\leftidx{^{\tilde{A}}}{\alpha_1}{}(x)\rightarrow\leftidx{^{\tilde{A}}}{\alpha_3}{}(x))$, for any $x$ and consequently $inf_x\{(\leftidx{^{\tilde{A}}}{\alpha_1}{}(x)\rightarrow\leftidx{^{\tilde{A}}}{\alpha_2}{}(x))\wedge(\leftidx{^{\tilde{A}}}{\alpha_2}{}(x)\rightarrow\leftidx{^{\tilde{A}}}{\alpha_3}{}(x))\}\leq inf_x\{\leftidx{^{\tilde{A}}}{\alpha_1}{}(x)\rightarrow\leftidx{^{\tilde{A}}}{\alpha_3}{}(x)\}$. Therefore,
\begin{align*}
& R_L(\leftidx{^{\tilde{A}}}{\alpha_1}{},\leftidx{^{\tilde{A}}}{\alpha_2}{})\wedge R_L(\leftidx{^{\tilde{A}}}{\alpha_2}{},\leftidx{^{\tilde{A}}}{\alpha_3}{})\\ & = inf_x\{\leftidx{^{\tilde{A}}}{\alpha_1}{}(x)\rightarrow \leftidx{^{\tilde{A}}}{\alpha_2}{}(x)\}\wedge inf_x\{\leftidx{^{\tilde{A}}}{\alpha_2}{}(x)\rightarrow \leftidx{^{\tilde{A}}}{\alpha_3}{}(x)\}\\
 & \leq inf_x\{(\leftidx{^{\tilde{A}}}{\alpha_1}{}(x)\rightarrow \leftidx{^{\tilde{A}}}{\alpha_2}{}(x))\wedge (\leftidx{^{\tilde{A}}}{\alpha_2}{}(x)\rightarrow \leftidx{^{\tilde{A}}}{\alpha_3}{}(x))\}\\
 & \leq inf_x\{\leftidx{^{\tilde{A}}}{\alpha_1}{}(x)\rightarrow\leftidx{^{\tilde{A}}}{\alpha_3}{}(x)\}\ \ \ \ \ [\text{using Property \ref{prop2}}]\\
 & = R_L(\leftidx{^{\tilde{A}}}{\alpha_1}{},\leftidx{^{\tilde{A}}}{\alpha_3}{}).
 \end{align*}
 
 \item $R_L(\leftidx{^{\tilde{A}}}{\alpha_1}{}\wedge\leftidx{^{\tilde{A}}}{\alpha_2}{},\leftidx{^{\tilde{A}}}{\alpha_1}{})=inf_x\{(\leftidx{^{\tilde{A}}}{\alpha_1}{}\wedge \leftidx{^{\tilde{A}}}{\alpha_2}{})(x)\rightarrow \leftidx{^{\tilde{A}}}{\alpha_1}{}(x)\}=1_L$, as $\leftidx{^{\tilde{A}}}{\alpha_1}{}\cap\leftidx{^{\tilde{A}}}{\alpha_2}{}\subseteq \leftidx{^{\tilde{A}}}{\alpha_1}{}$. Similarly $R_L(\leftidx{^{\tilde{A}}}{\alpha_1}{}\wedge\leftidx{^{\tilde{A}}}{\alpha_2}{},\leftidx{^{\tilde{A}}}{\alpha_2}{})=1_L$.
 
 \item $R_L(\leftidx{^{\tilde{A}}}{\alpha_1}{},\leftidx{^{\tilde{A}}}{0_L}{})=inf_x\{\leftidx{^{\tilde{A}}}{\alpha_1}{}(x)\rightarrow \leftidx{^{\tilde{A}}}{0_L}{}(x)\}=1_L$, as $\leftidx{^{\tilde{A}}}{0_L}{}(x)=\tilde{A}(x)$ and $\leftidx{^{\tilde{A}}}{\alpha_1}{}(x)\leq\tilde{A}(x)$, for any $x$.
 
 \item We have,
 \begin{align*}
 & R_L(\leftidx{^{\tilde{A}}}{\alpha_1}{},\leftidx{^{\tilde{A}}}{\alpha_2}{})\wedge R_L(\leftidx{^{\tilde{A}}}{\alpha_1}{},\leftidx{^{\tilde{A}}}{\alpha_3}{})\\
 & = inf_x\{\leftidx{^{\tilde{A}}}{\alpha_1}{}(x)\rightarrow \leftidx{^{\tilde{A}}}{\alpha_2}{}(x)\}\wedge inf_x\{\leftidx{^{\tilde{A}}}{\alpha_1}{}(x)\rightarrow\leftidx{^{\tilde{A}}}{\alpha_3}{}(x)\}\\
 & = inf_x\{(\leftidx{^{\tilde{A}}}{\alpha_1}{}(x)\rightarrow\leftidx{^{\tilde{A}}}{\alpha_2}{}(x))\wedge (\leftidx{^{\tilde{A}}}{\alpha_1}{}(x)\rightarrow \leftidx{^{\tilde{A}}}{\alpha_3}{}(x))\}\\
 & = inf_x\{\leftidx{^{\tilde{A}}}{\alpha_1}{}(x)\rightarrow (\leftidx{^{\tilde{A}}}{\alpha_2}{}(x)\wedge\leftidx{^{\tilde{A}}}{\alpha_3}{}(x))\}\ \ \ \ \ [\text{using Property \ref{prop5}}]\\
 & = inf_x\{\leftidx{^{\tilde{A}}}{\alpha_1}{}(x)\rightarrow (\leftidx{^{\tilde{A}}}{\alpha_2}{}\cap\leftidx{^{\tilde{A}}}{\alpha_3}{})(x)\}\\
 & = R_L(\leftidx{^{\tilde{A}}}{\alpha_1}{},\leftidx{^{\tilde{A}}}{\alpha_2}{}\cap\leftidx{^{\tilde{A}}}{\alpha_3}{}).
 \end{align*}
 
 \item Let $\leftidx{^{\tilde{A}}}{\alpha}{}\in \{\leftidx{^{\tilde{A}}}{\alpha_i}{}\}_i$, then 
 \begin{align*}
 R_L(\leftidx{^{\tilde{A}}}{\alpha}{},\bigcup_i\leftidx{^{\tilde{A}}}{\alpha_i}{}) & = inf_x\{\leftidx{^{\tilde{A}}}{\alpha}{}(x)\rightarrow (\bigcup_i\leftidx{^{\tilde{A}}}{\alpha_i}{})(x)\}\\
 & = inf_x\{\leftidx{^{\tilde{A}}}{\alpha}{}(x)\rightarrow \bigvee_i\leftidx{^{\tilde{A}}}{\alpha_i}{}(x)\}\\
 & = 1_L\ \ \ \ \ [\text{using Property \ref{prop7}}].
 \end{align*}
 
 \item Here we have,
 \begin{align*}
 inf_i\{R_L(\leftidx{^{\tilde{A}}}{\alpha_i}{},\leftidx{^{\tilde{A}}}{\alpha}{})\}
 & = inf_i\{inf_x\{\leftidx{^{\tilde{A}}}{\alpha_i}{}(x)\rightarrow \leftidx{^{\tilde{A}}}{\alpha}{}(x)\}\}\\
 & =  inf_x\{inf_i\{\leftidx{^{\tilde{A}}}{\alpha_i}{}(x)\rightarrow \leftidx{^{\tilde{A}}}{\alpha}{}(x)\}\}\\
 & = inf_x\{\bigvee_i(\leftidx{^{\tilde{A}}}{\alpha_i}{}(x))\rightarrow \leftidx{^{\tilde{A}}}{\alpha}{}(x)\} \ \ \ \ \ [using Property \ref{prop6}]\\
 & = inf_x\{(\bigcup_i\leftidx{^{\tilde{A}}}{\alpha_i}{})(x)\rightarrow \leftidx{^{\tilde{A}}}{\alpha}{}(x)\}\\
 & = R_L(\bigcup_i\leftidx{^{\tilde{A}}}{\alpha_i}{},\leftidx{^{\tilde{A}}}{\alpha}{}).
 \end{align*}
 
 \item From Theorem \ref{theorem4} we have $\leftidx{^{\tilde{A}}}{\alpha}{}\cap\bigcup_i\leftidx{^{\tilde{A}}}{\alpha_i}{}=\bigcup_i(\leftidx{^{\tilde{A}}}{\alpha}{}\cap\leftidx{^{\tilde{A}}}{\alpha_i}{})$. So,
 $R_L(\leftidx{^{\tilde{A}}}{\alpha}{}\cap\bigcup_i\leftidx{^{\tilde{A}}}{\alpha_i}{},\bigcup_i(\leftidx{^{\tilde{A}}}{\alpha}{}\cap\leftidx{^{\tilde{A}}}{\alpha_i}{}))=inf_x\{(\leftidx{^{\tilde{A}}}{\alpha}{}\cap\bigcup_i\leftidx{^{\tilde{A}}}{\alpha_i}{})(x)\rightarrow(\bigcup_i(\leftidx{^{\tilde{A}}}{\alpha}{}\cap\leftidx{^{\tilde{A}}}{\alpha_i}{}))(x)=1_L$.
 \end{enumerate}
 Hence $(\{\leftidx{^{\tilde{A}}}{\alpha}{}\mid \alpha \in L\},\leftidx{^{\tilde{A}}}{0_L}{}, \cap, \bigcup,R_L)$ is a localic frame.
\end{proof}
\begin{corollary}
$(\{\leftidx{^{\tilde{A}}}{\alpha}{}\mid \alpha \in [0,1]\},\leftidx{^{\tilde{A}}}{0}{}, \cap, \bigcup,R_{[0,1]})$ is a graded frame, where `$\rightarrow$' is the G$\ddot{o}$del arrow and $R_{[0,1]}(\leftidx{^{\tilde{A}}}{\alpha_1}{}, \leftidx{^{\tilde{A}}}{\alpha_2}{})=inf_x\{\leftidx{^{\tilde{A}}}{\alpha_1}{}(x)\rightarrow \leftidx{^{\tilde{A}}}{\alpha_2}{}(x)\}$ for $\alpha_1$, $\alpha_2\in [0,1]$.
\end{corollary}
It may be noted that the 5-tuple $(\{\leftidx{^{\tilde{A}}}{\alpha}{}\mid \alpha \in L\},\leftidx{^{\tilde{A}}}{0_L}{}, \cap, \bigcup,R_L)$ is a localic preordered set, where $L$ is a  frame, $R_L(\leftidx{^{\tilde{A}}}{\alpha_1}{}, \leftidx{^{\tilde{A}}}{\alpha_2}{})=inf_x\{\leftidx{^{\tilde{A}}}{\alpha_1}{}(x)\rightarrow \leftidx{^{\tilde{A}}}{\alpha_2}{}(x)\}$ for $\alpha_1$, $\alpha_2\in L$ and `$\rightarrow$' is the G$\ddot{o}$del-like arrow, as it satisfies all the properties to be a localic frame except the property namely $$R_L(\leftidx{^{\tilde{A}}}{\alpha_1}{},\leftidx{^{\tilde{A}}}{\alpha_2}{})\wedge R_L(\leftidx{^{\tilde{A}}}{\alpha_1}{},\leftidx{^{\tilde{A}}}{\alpha_3}{})=R_L(\leftidx{^{\tilde{A}}}{\alpha_1}{},\leftidx{^{\tilde{A}}}{\alpha_2}{}\cap \leftidx{^{\tilde{A}}}{\alpha_3}{}).$$
 \section{Some applications of fuzzy $\alpha$-cut}
 In this section we shall show some usages of the notion of fuzzy $\alpha$-cuts.
\subsection{Topological Structure}
Usually a fuzzy topological space is defined as a crisp set having fuzzy open sets \cite{CL, RL}. In 1992, Chakraborty and Ahsanullah proposed a notion of fuzzy topology on fuzzy sets \cite{MA}. This generalisation allows for defining topological subspaces on fuzzy subsets of the original fuzzy topological space. With respect to the classical definition subspaces have to be defined on crisp subsets of the original set. In our recent work on fuzzy topological systems \cite{MP2} we needed to use $\mathscr{L}$-topological spaces where the value set $L$ is a frame. This is one further step towards generalisation of \cite{MA}. To make this paper self contained we give the definition below. 
\begin{definition}[$\mathscr{L}$-Topological Space]\cite{MA}\label{ltop}
 Let $(X,\tilde{A})$ be an $L$-fuzzy set and $\tau$ a collection of fuzzy subsets of $(X, \tilde{A})$ such that
 \begin{enumerate}
\item $(X,\tilde{\emptyset})$ and $(X,\tilde{A})$ are in $\tau$, where $\tilde{\emptyset} : X \longrightarrow L$ is such that $\tilde{\emptyset}(x) = 0_L$, for all $x \in X$, where $0_L$ is the least element of the frame $L$;
\item $(X,\tilde{A_1})$, $(X,\tilde{A_2})$ are in $\tau$ implies $(X,\tilde{A_1} \cap\tilde{A_2})$ is in $\tau$, where $(\tilde{A_1} \cap\tilde{A_2})(x)$ = $\tilde{A_1}(x) \wedge \tilde{A_2}(x)$, for all $x \in X$;
\item $(X,\tilde{A_i})\in\tau$ implies $(X,\bigvee_{i\in I}\tilde{A_i}) \in \tau$, where $\bigvee_{i\in I}\tilde{A_i} : X \longrightarrow L$ is such that $(\bigcup_{i\in I} \tilde{A_i})(x) = \bigvee_{i\in I} \tilde{A_i}(x)$, for all $x \in X$.
 \end{enumerate}
Then $(X, \tilde{A}, \tau)$ is an $\mathscr{L}$-topological space.
 \end{definition}
 One can easily see that the fuzzy $\alpha$-cuts of a fuzzy topological space are fuzzy topological subspaces. More specifically, we have.
 \begin{theorem}\cite{MP2}
 Let $(X,\tilde{A},\tau)$ be an $\mathscr{L}$-topological space and a fuzzy $\alpha$-cut of $(X,\tilde{A})$ i.e., $(X,\leftidx{^{\tilde{A}}}{\alpha}{})$ be taken. Let $\tau'$ be defined by $\tau' = \{(X,\tilde{T'}) \mid \tilde{T'} =\leftidx{^{\tilde{A}}}{\alpha}{} \cap\tilde{T},\tilde{T}\in\tau\}$. Then $(X,\leftidx{^{\tilde{A}}}{\alpha}{},\tau')$ also forms an $\mathscr{L}$-topological space and is an $\mathscr{L}$-topological subspace.
\end{theorem}
Hence fuzzy $\alpha$-cuts provide us with a natural class of fuzzy substructure of $\mathscr{L}$-topological space.

It should be noted that in subsequent years there has been a lot of serious work on fuzzy topological spaces from the angle of category theory \cite{MA, MB1, UH, RL}. But to our knowledge, the notion of fuzzy $\alpha$-cuts as substructures has not been discussed. It would be interesting to investigate what kind of sub objects these fuzzy $\alpha$-cuts give rise to.
\subsection{Algebraic Structure}
A similar approach was initially adopted in developing fuzzy algebraic structures. Rosenfeld's pioneering work in fuzzy groups starts with an ordinary group and proceeds to define fuzzy subgroups of that group. On the other hand in \cite{MB}, Chakraborty and Banerjee defined fuzzy operations on fuzzy sets thus obtaining a generalisation that was intended. They, however, placed their work in categorical framework. We shall adopt their idea basically but avoiding categorical language and then show the role of fuzzy $\alpha$-cuts in this context. It is to be noted that a binary operation on a crisp set $A$ (e.g. the group operation) is a mapping from $A\times A$ to $A$. We shall define a fuzzy binary operation on a fuzzy set $(X,\tilde{A})$ using fuzzy equality. It is also to be noted that the Cartesian product of two $L$-fuzzy sets $(X,\tilde{A})$ and $(Y,\tilde{B})$ is the $L$-fuzzy set $(X\times Y,\tilde{A}\times \tilde{B})$ where $(\tilde{A}\times\tilde{B})(x,y)=\tilde{A}(x)\wedge\tilde{B}(y)$. So a fuzzy binary composition on $(X,\tilde{A})$ has to be a kind of mapping from $(X\times X,\tilde{A}\times\tilde{A})$ to $(X,\tilde{A})$ where the pre-image is mapped to the image to some grade belonging to the frame $L$. For any $L$-fuzzy set $(X,\tilde{A})$ by $\mid\tilde{A}\mid$ is meant the support viz. $\{x\in X\mid\tilde{A}(x)>0_L\}$. Instead of using pre-fix notation for the operator $\oplus$ we shall use infix notation, i.e. we write 
$$x_1\oplus x_2=x_3,\ \text{for}\ \oplus(x_1,x_2)=x_3$$
and additionally equality relation $(=)$ is graded. That is, for any $x\in X$, $y\in Y$, the expression $x=y$ gets a degree from $L$. 
We shall write $gr(x\simeq y)$ to make a distinction between fuzzy equality and ordinary equality. Formally, we have Definition \ref{fbo}.
\begin{definition}[$L$-fuzzy binary operation]\label{fbo}
An $L$-fuzzy binary operation $\oplus:(X\times X,\tilde{A}\times\tilde{A})\longrightarrow (X,\tilde{A})$ is a map such that
\begin{enumerate}
\item $gr(x_1\oplus x_2\simeq x_3)\leq \tilde{A}\times \tilde{A}(x_1,x_2)\wedge\tilde{A}(x_3)$ for any $x_1,x_2,x_3\in X$, where $gr:(X\times X)\times X\longrightarrow L$ and $\tilde{A}\times\tilde{A}(x_1,x_2)=\tilde{A}(x_1)\wedge\tilde{A}(x_2)$;
\item for any $(a_1,a_2)\in\mid\tilde{A}\times\tilde{A}\mid$ there exist a unique $a\in\mid\tilde{A}\mid$ with $gr(a_1\oplus a_2\simeq a)=\tilde{A}(a_1)\wedge\tilde{A}(a_2)$ and $gr(a_1\oplus a_2\simeq a')=0_L$, if $a'(\neq a)\in \mid\tilde{A}\mid$.
\end{enumerate}
\end{definition}
Note that condition 1 is the fuzzy counterpart of the closure property of $\tilde{A}$ relative to the operation $\oplus$. Also $gr(x_1\oplus x_2\simeq x_3)$ represents the fuzzy equality that is, the grade in which the pair $(x_1,x_2)$ equals to $x_3$ by the fuzzy composition $\oplus$. That is, here we will talk about the degree of equality between $x_1\oplus x_2$ and $x_3$ of $X$. 
\begin{definition}[$L$-fuzzy group]
An $L$-fuzzy group is a triple $(X,\tilde{A},\oplus)$ consisting of an $L$-fuzzy set $(X,\tilde{A})$ with $\mid\tilde{A}\mid\neq\emptyset$ and an $L$-fuzzy binary operation $\oplus$ such that
\begin{enumerate}
\item if for any $a,b,a_1,a_2,a_3,b_1,b_2\in \mid\tilde{A}\mid$,
\begin{align*}
gr(a_1\oplus a_2\simeq b_1) & = \tilde{A}(a_1)\wedge \tilde{A}(a_2),\\
gr(b_1\oplus a_3\simeq a) & = \tilde{A}(b_1)\wedge \tilde{A}(a_3),\\
gr(a_2\oplus a_3\simeq b_2) & = \tilde{A}(a_2)\wedge \tilde{A}(a_3) \ and\\
gr(a_1\oplus b_2\simeq b) & = \tilde{A}(a_1)\wedge \tilde{A}(b_2)\\
\end{align*}
then $a=b$;
\item there exist $e\in\mid\tilde{A}\mid$ such that $gr(a\oplus e\simeq a)=\tilde{A}(a)\wedge\tilde{A}(e)=gr(e\oplus a\simeq a)$ for any $a\in \mid\tilde{A}\mid$ and
\item for any $a\in \mid\tilde{A}\mid$, there exist $a^{-1}\in \mid\tilde{A}\mid$, such that $gr(a\oplus a^{-1}\simeq e)=\tilde{A}(a)\wedge\tilde{A}(a^{-1})=gr(a^{-1}\oplus a\simeq e)$ and $\tilde{A}(a)=\tilde{A}(a^{-1})$.
\end{enumerate}
$e\in\mid\tilde{A}\mid$ described in 2 is known as the identity whereas $a^{-1}\in \mid\tilde{A}\mid$ for each $a\in \mid\tilde{A}\mid$ illustrated in condition 3 are known as inverse of $a$ in the $L$-fuzzy group $(X,\tilde{A},\oplus)$. It is possible to show that identity and inverse of an element in the $L$-fuzzy group are unique.

It is to be noted that property 1 of being an $L$-fuzzy group represents the fuzzy version of associativity.
\end{definition}
Equipped with this definition of an $L$-fuzzy group, it will be observed how does fuzzy $\alpha$-cuts play a role.
\begin{definition}[$L$-fuzzy subgroup]
Let $(X,\tilde{A},\oplus)$ be an $L$-fuzzy group and $(X,\tilde{B})$ be a $L$-fuzzy subset of the $L$-fuzzy set $(X,\tilde{A})$. Then the fuzzy substructure $(X,\tilde{B},\oplus')$ where $\oplus'$ is an $L$-fuzzy binary operation on $(X,\tilde{B})$ defined by, $$gr(x_1\oplus'x_2\simeq x_3)=gr(x_1\oplus x_2\simeq x_3)\wedge\tilde{B}(x_1)\wedge\tilde{B}(x_2)\wedge\tilde{B}(x_3),$$ for any $x_1,x_2,x_3\in X$
is called a fuzzy subgroup of $(X,\tilde{A},\oplus)$ if $(X,\tilde{B},\oplus')$ is itself an $L$-fuzzy group, [$\oplus'$ is the restriction of $\oplus$ on the fuzzy subset $(X,\tilde{B})$ of $(X,\tilde{A})$].
\end{definition}

\begin{proposition}\label{iden}
Let $(X,\tilde{A},\oplus)$ be an $L$-fuzzy group. Then for any $a\in \mid\tilde{A}\mid$, $\tilde{A}(a)\leq \tilde{A}(e)$.
\end{proposition}
\begin{proof}
For any $a\in\mid\tilde{A}\mid$, there exist $a^{-1}\in\mid\tilde{A}\mid$ such that $gr(a\oplus a^{-1}\simeq e)=\tilde{A}(a)\wedge \tilde{A}(a^{-1})=\tilde{A}(a)$, as $\tilde{A}(a)=\tilde{A}(a^{-1})$. Also we know $gr(a\oplus a^{-1}\simeq e)\leq \tilde{A}(e)$ and consequently $\tilde{A}(a)\leq\tilde{A}(e)$, for any $a\in\mid\tilde{A}\mid$.
\end{proof}

\begin{theorem}
Let $(X,\tilde{A},\oplus)$ be an $L$-fuzzy group. Then  $(X,\leftidx{^{\tilde{A}}}{\alpha}{},\oplus')$ is an $L$-fuzzy sub group for any $\alpha\in L$ such that $\mid \leftidx{^{\tilde{A}}}{\alpha}{}\mid\neq\emptyset$.
\end{theorem}
\begin{proof}
Given that $(X,\tilde{A},\oplus)$ is an $L$-fuzzy group. Let us restrict the function $\oplus$ on $\leftidx{^{\tilde{A}}}{\alpha}{}$. Then for $x,y,z\in X$, if $gr(x\oplus y\simeq z)=\tilde{A}(x)\wedge \tilde{A}(y)$ then 
$$gr(x\oplus' y\simeq z)=\tilde{A}(x)\wedge\tilde{A}(y)\wedge \leftidx{^{\tilde{A}}}{\alpha}{}(x)\wedge \leftidx{^{\tilde{A}}}{\alpha}{}(y)\wedge \leftidx{^{\tilde{A}}}{\alpha}{}(z).$$
Now as $\tilde{A}(x)\wedge\tilde{A}(y)\leq\tilde{A}(z)$, we have $\leftidx{^{\tilde{A}}}{\alpha}{}(x)\wedge\leftidx{^{\tilde{A}}}{\alpha}{}(y)\leq\leftidx{^{\tilde{A}}}{\alpha}{}(z).$ Therefore $\leftidx{^{\tilde{A}}}{\alpha}{}(x)\wedge\leftidx{^{\tilde{A}}}{\alpha}{}(y)\wedge\leftidx{^{\tilde{A}}}{\alpha}{}(z)=\leftidx{^{\tilde{A}}}{\alpha}{}(x)\wedge\leftidx{^{\tilde{A}}}{\alpha}{}(y)$. Hence for any $x,y,z\in X$ if $gr(x\oplus y\simeq z)=\tilde{A}(x)\wedge\tilde{A}(y)$ then $gr(x\oplus' y\simeq z)=\leftidx{^{\tilde{A}}}{\alpha}{}(x)\wedge\leftidx{^{\tilde{A}}}{\alpha}{}(y)$.

To show that $(X,\leftidx{^{\tilde{A}}}{\alpha}{},\oplus')$ is an $L$-fuzzy subgroup for $\alpha$ such that $\mid\leftidx{^{\tilde{A}}}{\alpha}{}\mid\neq\emptyset$ first of all notice that $\oplus'$ is indeed an $L$-fuzzy binary operation as the following holds.

(i) $gr(x\oplus'y\simeq z)=gr(x\oplus y\simeq z)\wedge\leftidx{^{\tilde{A}}}{\alpha}{}(x)\wedge\leftidx{^{\tilde{A}}}{\alpha}{}(y)\wedge\leftidx{^{\tilde{A}}}{\alpha}{}(z)\leq(\leftidx{^{\tilde{A}}}{\alpha}{}(x)\wedge\leftidx{^{\tilde{A}}}{\alpha}{}(y))\wedge\leftidx{^{\tilde{A}}}{\alpha}{}(z)$.

(ii) If $a,b\in\mid\leftidx{^{\tilde{A}}}{\alpha}{}\mid$ then $\leftidx{^{\tilde{A}}}{\alpha}{}(a)>0_L$, $\leftidx{^{\tilde{A}}}{\alpha}{}(b)>0_L$. Hence $\tilde{A}(a)\geq \alpha$ and $\tilde{A}(b)\geq\alpha$. As $\oplus$ is an $L$-fuzzy binary operation so for any $a,b\in\mid\leftidx{^{\tilde{A}}}{\alpha}{}\mid\subseteq\mid\tilde{A}\mid$ there exist unique $c\in\mid\tilde{A}\mid$ with $gr(a\oplus b\simeq c)=\tilde{A}(a)\wedge\tilde{A}(b)\leq\tilde{A}(c)$ and $gr(a\oplus b\simeq c')=0_L$, if $c'(\neq c)\in\mid\tilde{A}\mid$. Therefore $\tilde{A}(c)\geq \alpha$ and hence $\leftidx{^{\tilde{A}}}{\alpha}{}(c)=\tilde{A}(c)$. $c\in \mid\leftidx{^{\tilde{A}}}{\alpha}{}\mid$. Consequently for $(a,b)\in\mid\leftidx{^{\tilde{A}}}{\alpha}{}\times\leftidx{^{\tilde{A}}}{\alpha}{}\mid$, there exist unique $c\in\mid\leftidx{^{\tilde{A}}}{\alpha}{}\mid$ with $gr(a\oplus' b\simeq c)=\leftidx{^{\tilde{A}}}{\alpha}{}(a)\wedge\leftidx{^{\tilde{A}}}{\alpha}{}(b)$ and $gr(a\oplus' b\simeq c')=0_L$ if $c'(\neq c)\in\mid\leftidx{^{\tilde{A}}}{\alpha}{}\mid$.

As for any $x,y,z\in X$ if $gr(x\oplus y\simeq z)=\tilde{A}(x)\wedge\tilde{A}(y)$ then $gr(x\oplus' y\simeq z)=\leftidx{^{\tilde{A}}}{\alpha}{}(x)\wedge\leftidx{^{\tilde{A}}}{\alpha}{}(y)$, associativity holds good.

Now $(X,\tilde{A},\oplus)$ is an $L$-group and so there exist $e\in\mid\tilde{A}\mid$ such that for any $a\in \mid\tilde{A}\mid$, $gr(a\oplus e\simeq a)=\tilde{A}(a)\wedge\tilde{A}(e)$.
Hence for any $a\in\mid\leftidx{^{\tilde{A}}}{\alpha}{}\mid\subseteq\mid\tilde{A}\mid$, $gr(a\oplus' e\simeq a)=\leftidx{^{\tilde{A}}}{\alpha}{}(a)\wedge\leftidx{^{\tilde{A}}}{\alpha}{}(e)$. Now $a\in\mid \leftidx{^{\tilde{A}}}{\alpha}{}\mid$ implies $\leftidx{^{\tilde{A}}}{\alpha}{}(a)>0_L$ and hence $\tilde{A}(a)\geq \alpha$. Using Proposition \ref{iden}, we have $\alpha\leq\tilde{A}(a)\leq\tilde{A}(e)$. Hence $\leftidx{^{\tilde{A}}}{\alpha}{}(e)=\tilde{A}(e)$. So, $e\in\mid\leftidx{^{\tilde{A}}}{\alpha}{}\mid$.

Similarly it can be shown that for any $a\in\mid\leftidx{^{\tilde{A}}}{\alpha}{}\mid$, there exist $a^{-1}\in\mid\leftidx{^{\tilde{A}}}{\alpha}{}\mid$ such that $gr(a\oplus' a^{-1}\simeq e)=\leftidx{^{\tilde{A}}}{\alpha}{}(a)\wedge \leftidx{^{\tilde{A}}}{\alpha}{}(a^{-1})$, as $\tilde{A}(a^{-1})=\tilde{A}(a)\geq \alpha$.
\end{proof}
Thus fuzzy $\alpha$-cuts form natural fuzzy subgroups of the fuzzy group.

This method of defining fuzzy algebraic structures and their sub structures may be adopted for any kind of algebraic structure not necessarily fuzzy groups only.
\begin{example}
Let $X=\{x_1,x_2,x_3,x_4,x_5\}$, $\tilde{A}:X\longrightarrow L$ where
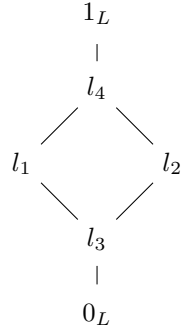
\begin{center}
\begin{tikzpicture}
  \matrix (galois)
     [matrix of nodes,%
      nodes in empty cells,
      nodes={outer sep=0pt,circle,minimum size=4pt},
      column sep={1cm,between origins},
      row sep={1cm,between origins}]
   {
    && $1_L$ &&\\
    &  & $l_4$ & \\
   & $l_1$ &  & $l_2$\\
    && $l_3$ &&\\
    && $0_L$ &&\\
   };
      \draw (galois-1-3) -- (galois-2-3);
      \draw (galois-2-3) -- (galois-3-2);
      \draw (galois-2-3) -- (galois-3-4);
       \draw (galois-3-2) -- (galois-4-3);
       \draw (galois-3-4) -- (galois-4-3);
        \draw (galois-4-3) -- (galois-5-3);

\end{tikzpicture} 
\end{center}
such that $\tilde{A}(x_i)=l_i$, for $i=\{1,2,3,4\}$ and $\tilde{A}(x_5)=0_L$. Here $\mid\tilde{A}\mid=\{x_1,x_2,x_3,x_4\}$. Let us define $\oplus$ as follows:
\begin{align*}
gr(x_1\oplus x_1\simeq x_4)=l_1 && gr(x_1\oplus x_2\simeq x_3)=l_3\\  
gr(x_1\oplus x_3\simeq x_2)=l_3 && gr(x_1\oplus x_4\simeq x_1)=l_1 \\
gr(x_2\oplus x_1\simeq x_3)=l_3 && gr(x_2\oplus x_2\simeq x_4)=l_2 \\
gr(x_2\oplus x_3\simeq x_1)=l_3 && gr(x_2\oplus x_4\simeq x_2)=l_2 \\
gr(x_3\oplus x_1\simeq x_2)=l_3 && gr(x_3\oplus x_2\simeq x_1)=l_3 \\
gr(x_3\oplus x_3\simeq x_4)=l_3 && gr(x_3\oplus x_4\simeq x_3)=l_3 \\
gr(x_4\oplus x_1\simeq x_1)=l_1 && gr(x_4\oplus x_2\simeq x_2)=l_2 \\
gr(x_4\oplus x_3\simeq x_3)=l_3 && gr(x_4\oplus x_4\simeq x_4)=l_4 \\
\end{align*}
Then clearly $x_4=e$, $x_1^{-1}=x_1$, $x_2^{-1}=x_2$, $x_3^{-1}=x_3$ and $x_4^{-1}=x_4$. Hence $(X,\tilde{A},\oplus)$ is an $L$-fuzzy group.

Let $\alpha=l_1$, then $\mid\leftidx{^{\tilde{A}}}{l_1}{}\mid=\{x_1,x_4\}\neq\emptyset$ and $(X,\leftidx{^{\tilde{A}}}{l_1}{},\oplus')$ forms an $L$-fuzzy group and consequently becomes an $L$-fuzzy subgroup of $(X,\tilde{A},\oplus)$. Similarly for other $\alpha\in L$, where $\mid\leftidx{^{\tilde{A}}}{\alpha}{}\mid\neq\emptyset$, it can be shown that $(X,\leftidx{^{\tilde{A}}}{\alpha}{},\oplus')$ forms $L$-fuzzy subgroups of $(X,\tilde{A},\oplus)$.
\end{example}
\subsection{Probabilistic Rough Set Theory}
We will now observe another kind of usefulness of the notion of fuzzy $\alpha$-cuts in the context of rough set theory \cite{PW, YY}. An approximation space is a tuple $(X,R)$, consisting of a set of objects $X$ and an equivalence relation $R$, known as indiscernibility relation on $X$. For any $A\subseteq X$, the lower and upper approximations of $A$ in the approximation space $(X,R)$ are denoted by $\underline{A}$ and $\overline{A}$ respectively and defined as follows.
$$\underline{A}=\bigcup\{[x]\mid [x]\subseteq A\};$$
$$\overline{A}=\bigcup\{[x]\mid A\cap [x]\neq\emptyset\}.$$
A rough membership function of $A$, denoted by $\mu_A$, is a function from $X$ to $[0,1]$ such that $\mu_A(x)=\frac{\mid [x]\cap A\mid}{\mid [x]\mid}\leq 1$, where $\mid S\mid$ stands for the cardinality of the set $S$ and $[x]$ stands for the equivalence class of $x\in X$. In this definition $X$ is taken to be a finite set.

In \cite{YY}, we notice that for generalised probabilistic approximations, they considered a pair of parameters $\alpha,\beta\in [0,1]$ with $\alpha\geq\beta$ to ensure that the lower approximation is smaller than the upper approximation in order to be consistent with existing approximation operators.

In the theory of probabilistic rough sets a weight or grade from the set $[0,1]$ is attached with each granule. The grades of granules are obtained with the help of some rough membership function. In particular the grade of granule may be determined with the help of above described rough membership function. Notice that in \cite{YY}, while defining lower and upper approximations of a set $A$, $\alpha$-cuts and strict $\beta$-cuts are used with $0\leq\beta<\alpha\leq 1$ in the following way.
$$\underline{A_\alpha}=\{x\in X\mid\mu_A(x)\geq \alpha\};$$ 
$$\overline{A_\beta}=\{x\in X\mid\mu_A(x)> \beta\}.$$
These are crisp sets. Hence the grade disappears in the final approximations.

But while defining lower and upper approximations of a set $A$, if we use the concept of fuzzy $\alpha$-cuts and fuzzy $\beta$-cuts instead of $\alpha$-cuts and strict $\beta$-cuts then we will able to end up with the final approximations having grades. That is, the lower and upper approximations becomes fuzzy sets and defined as follows.
$$\underline{\leftidx{^{A}}{\alpha}{}}:X\longrightarrow [0,1]\ \text{s.t.}\ \underline{\leftidx{^{A}}{\alpha}{}}(x)=
\begin{cases}
\mu_A(x) & \text{if}\ \mu_A(x)\geq \alpha\\
0 & \text{otherwise}.
\end{cases}$$
$$\overline{\leftidx{^{A}}{\beta}{}}:X\longrightarrow [0,1]\ \text{s.t.}\ \overline{\leftidx{^{A}}{\beta}{}}(x)=
\begin{cases}
\mu_A(x) & \text{if}\ \mu_A(x)\geq \beta\\
0 & \text{otherwise}.
\end{cases}$$
It is quite expected that the above described notion of lower and upper approximations will play a significant role in probabilistic rough set theory. Here instead of two different types of cuts viz. $\alpha$-cuts and strict $\beta$-cuts one type of cut has been used uniformly in determining lower and upper approximations. In this paper we will not delve into this topic, but it will be considered in our future research.
\section{Concluding Remarks}
In this paper we have dealt with the notion of fuzzy $\alpha$-cut and its significance. Study of the family of fuzzy $\alpha$-cuts provides an example of graded frame which was introduced in \cite{MP3}. Moreover in this paper we generalise the notion of graded frame one step further and call it `semilinear frame'. It is to emphasise that the notion semilinearity introduced in this paper is more general than `prelinearity'; while the latter notion has been widely discussed in literature, the former notion is not. We also proposed the notion of localic frame in this work. A detailed study of G$\ddot{o}$del-like arrow provides a nice result about the relation between prelinearity and semilinearity property. The algebraic notion of semilinear frame needs to be studied in more detail. Taking a general fuzzy arrow instead of G$\ddot{o}$del arrow may also be considered as an interesting future project.

\paragraph{Acknowledgements.} 
The corresponding author acknowledge Department of Science $\&$ Technology, Government of India for financial support vide reference no. SR/WOS-A/PM-1010/2014 under Women Scientist Scheme to carry out this work.

\section*{References}

\bibliography{mybibfile}

\begin{thebibliography}{10} 
\bibitem{AK1} K. Atanasov, \emph{Intuitionistic Fuzzy Sets}, Fuzzy Sets and Systems, \textbf{20}, no. 1, pp. 87--96 (1986).
\bibitem{RB} R. B{\u e}lohl{\' a}vek, \emph{Fuzzy Relational Systems: Foundations and Principles}, Kluwer Academic Publishers, New York, 2002.
\bibitem {BW} W.D. Blizard, \emph{Real-valued Multisets and Fuzzy Sets}, Fuzzy Sets and Systems, \textbf{33}, pp. 77--97 (1989). 
\bibitem{MA} M.K. Chakraborty and T.M.G. Ahsanullah, \emph{Fuzzy topology on fuzzy sets and tolerance topology}, Fuzzy Sets and Systems, \textbf{45}, pp. 103--108 (1992).
\bibitem{MB1} M. K. Chakraborty and M. Banerjee, \emph{A new category for fuzzy topological spaces}, Fuzzy Sets and Systems, \textbf{51}, 1992, pp. 227--233.
\bibitem{MB} M. K. Chakraborty and M. Banerjee \emph{A categorical approach to fuzzy set theory}, Recent advances in fuzzy mathematics. In: Proceedings of the national seminar on fuzzy mathematics and its applications, Tripura University, Tripura (1991).
\bibitem{MP3} M. K. Chakraborty and P. Jana, \emph{Fuzzy topology via fuzzy geometric logic with graded consequence}, International Journal of Approximate Reasoning, \textbf{80}, pp. 334--347 (2017).  
\bibitem{CL} C. L. Chang, \emph{Fuzzy topological spaces}, J. Math. Anal. Appl., \textbf{24} pp. 182--190 (1968).
\bibitem{CE} E.W. Chapin, \emph{Set-valued Set Theory}, I, Notre Dame J. Formal Logic, \textbf{15}, pp. 619--634 (1974). 
\bibitem{JA} J.T. Denniston, A. Melton, S.E. Rodabaugh, and  S.A. Solovyov, \emph{Lattice-valued preordered sets as lattice-valued topological systems}, Fuzzy Sets and Systems, \textbf{259}, pp. 89--110 (2015). 
\bibitem{MS} S. Dutta, M. K. Chakraborty, \emph{Fuzzy relation, fuzzy function over fuzzy sets: a retrospective}, \textbf{19}(1), pp. 99--112 (2015). 
\bibitem{GJ} J.A. Goguen, \emph{L-fuzzy sets}, Journal of Mathematical Analysis and Applications \textbf{18}, pp. 145--174 (1967).
\bibitem{PH} P. H{\'a}jek, \emph{Metamathematics of Fuzzy Logic}, Kluwer Academic Publishers (1998).
\bibitem{UH} U. H{\"o}hle, \emph{Fuzzy topologies and topological space objects in a topos}, Fuzzy Sets and Systems, \textbf{19}, 1986, pp. 299--304.
\bibitem{MP2} P. Jana, M.K. Chakraborty, \emph{Categorical relationships of fuzzy topological systems with fuzzy topological spaces and underlying algebras-II},  Ann. of Fuzzy Math. and Inform., \textbf{10}, no. 1, pp. 123--137 (2015).
\bibitem{KY} G.J. Klir, B. Yuan, \emph{Fuzzy Sets and Fuzzy Logic: Theory and Applications}, Prentice Hall Publishers (1995).
\bibitem{RL} R. Lowen, \emph{Fuzzy topological spaces and fuzzy compactness}, J. Math. Anal. Appl., \textbf{56} pp. 621--633 (1976).
\bibitem{PW} Z. Pawlak, \emph{Rough sets}, International Journal of Computer and Information Sciences, \textbf{11}, pp. 341--356 (1982).
\bibitem{TR} T. Radecki, \emph{Level Fuzzy Sets}, Journal of Cybernetics, \textbf{7}, pp. 189--198 (1977).
\bibitem{AR} A. Rosenfeld \emph{Fuzzy groups}, J. Math. Anal. Appl., \textbf{35}, pp. 512--517 (1971).
\bibitem{SV} S. J. Vickers, \emph{Topology Via Logic}, volume 5, Cambridge Tracts Theoret. Comput. Sci., 1989.
\bibitem{YY} Y. Yao, \emph{Probabilistic rough set approximations}, International Journal of Approximate Reasoning, \textbf{49}, pp. 255--271 (2007).
\bibitem{LZ} L.A. Zadeh, \emph{Fuzzy sets}, Information Control, \textbf{8}, pp. 338--353 (1995).
\bibitem{LZ1} L.A. Zadeh, \emph{Similarity relations and fuzzy orderings}, Information Sciences, \textbf{3(2)}, pp. 177--200 (1971).
\end{thebibliography}

\end{document}